\documentclass{amsart}
\usepackage{amsmath,amscd,amssymb,amsthm,graphicx,xcolor}

\theoremstyle{plain}
\newtheorem{lemma}{Lemma}[section]
\newtheorem{prop}[lemma]{Proposition}
\newtheorem{thm}[lemma]{Theorem}
\newtheorem{cor}[lemma]{Corollary}

\theoremstyle{definition}
\newtheorem*{defn}{Definition}
\newtheorem*{rem}{Remark}

\def\ts{\textstyle}

\def\vtsp{\hspace{.1em}}

\def\cal{\mathcal}

\def\A{{\cal A}}

\def\F{{\cal F}}

\def\I{{\cal I}}
\def\J{{\cal J}}
\def\K{{\cal K}}
\def\Koo{\K_{1|1}}

\def\L{{\cal L}}

\def\S{{\cal S}}

\def\frak{\mathfrak}

\def\da{{\frak a}}

\def\db{{\frak b}}

\def\dc{{\frak c}}

\def\dg{{\frak g}}

\def\dh{{\frak h}}

\def\dl{{\frak l}}

\def\dgo{{\frak o}}
\def\osp{\dgo\ds\dgp}

\def\dgp{{\frak p}}

\def\ds{{\frak s}}
\def\dsl{\ds\dl}

\def\dU{{\frak U}}

\def\Bbb{\mathbb}

\def\bC{\Bbb C}

\def\bN{\Bbb N}

\def\bR{\Bbb R}
\def\Roo{\bR^{1|1}}

\def\bZ{\Bbb Z}

\def\la{\langle}

\def\o{\overline}
\def\ot{\otimes}
\def\ra{\rangle}

\def\t{\tilde}

\def\acong{\buildrel\da\over\cong}

\def\bcong{\buildrel\db\over\cong}

\def\ad{\mathop{\rm ad}\nolimits}
\def\Ann{\mathop{\rm Ann}\nolimits}

\def\dim{\mathop{\rm dim}\nolimits}
\def\End{\mathop{\rm End}\nolimits}

\def\ker{\mathop{\rm ker}\nolimits}

\def\mult{\mathop{\rm mult}\nolimits}

\def\oh{{\ts\frac{1}{2}}}
\def\proj{\mathop{\rm proj}\nolimits}
\def\px{\partial_x}

\def\res{\mathop{\rm res}\nolimits}

\def\spanrm{\mathop{\rm span}\nolimits} 

\def\sym{\mathop{\rm sym}\nolimits}

\def\thup{{\mathrm{th}}}

\def\VR{\mathrm{Vec} \vtsp \bR}
\def\VRm{\mathrm{Vec} \vtsp \bR^m}

\def\dog{differential operator}

\def\iff{if and only if}

\def\irr{irreducible}

\def\lsa{Lie superalgebra}

\def\lwv{lowest weight vector}

\def\r{representation}  

\def\tdm{tensor density module}
\def\tfm{tensor field module}
\def\uea{universal enveloping algebra}
\def\vf{vector field}

\title[Vector field annihilators]{An approach to annihilators in the context of vector field Lie algebras}

\author{Charles H.\ Conley}
\address{Department of Mathematics \\University of North Texas \\Denton TX 76203, USA} 
\email{conley@unt.edu}

\author{William Goode}
\address{Department of Mathematics \\Vanderbilt University \\Nashville TN 37240, USA} 
\email{william.m.goode@vanderbilt.edu}

\thanks{The first author was partially supported by Simons Foundation Collaboration Grant 519533.}

\begin{document}

\begin{abstract}
We present a general method for describing the annihilators of modules of Lie algebras under certain conditions, which hold for some tensor modules of vector field Lie algebras.  As an example, we apply the method to obtain an efficient proof of previously known results on the annihilators of the bounded irreducible modules of $\VR$.
\end{abstract}

\maketitle

\thispagestyle{empty}

\subsubsection*{Dedication}
As a student and a ``grandstudent'' of V.~S.~Varadarajan, we dedicate this article to his memory.  To those of us who were his students, Raja was both teacher and friend.  His generosity and his passion for doing mathematics elegantly have been a lasting inspiration, which we have sought to communicate forward to his academic descendants.

\tableofcontents

\section{Introduction} \label{Intro}
Suppose that $\pi: \dg \to \End(V)$ is a \r\ of a Lie algebra $\dg$ on a vector space $V$.  Write $\pi|_{\dU(\dg)}: \dU(\dg) \to \End(V)$ for its extension to a \r\ of the \uea\ $\dU(\dg)$.

\begin{defn}
The \textit{annihilator} of $\pi$ is the kernel of $\pi|_{\dU(\dg)}$, a two-sided ideal in $\dU(\dg)$.  We denote it by $\Ann_\dg (V)$:
\begin{equation*}
   \Ann_\dg (V) := \ker \big( \pi|_{\dU(\dg)} \big).
\end{equation*}
\end{defn}

In this article we present a method for describing $\Ann_\dg (V)$ under certain conditions on $\dg$ and $\pi$.  It is completely elementary, relying solely on linear algebra.  The basic version applies only to annihilators generated in a single degree, but it can be modified to apply to annihilators generated in multiple degrees.  The conditions are restrictive, but they are satisfied by the \irr\ admissible modules of the Lie algebra of \vf s on the line, $\VR$.  Moreover, it may be that they are satisfied by a reasonably general class of modules of \vf\ Lie algebras, one which includes the tensor field modules of $\VRm$.

The \irr\ admissible modules of $\VR$ (as well as those of the Virasoro Lie algebra) were classified by Mathieu \cite{Ma92}.  They are the \tdm s, which are deformations of the standard module of $\VR$ on the functions $\bC[x]$.  The single degree version of our method applies to $\bC[x]$, and also to the direct sum of all of the \tdm s.  It shows that the annihilators of both of these modules are generated in degree~$2$.  Note that the annihilator of the direct sum is simply the intersection of the annihilators of all of the \tdm s, and so by Mathieu's result it may be thought of as the ``admissible Jacobson radical'' of $\dU(\VR)$.

These two annihilators were originally considered by Martin and the first author, using less efficient techniques \cite{CM07}.  They observed that with the description of the admissible Jacobson radical in hand, it is easy to describe the annihilators of each of the individual \tdm s: those other than that of $\bC[x]$ itself turn out to be generated in degrees $2$ and~$3$.  We will see that the multidegree version of our method gives a way to deduce these annihilators directly.

The second author wrote his dissertation on the analogous project for $\Koo$, the \lsa\ of contact \vf s on the superline $\Roo$ \cite{Go23}.  There the techniques of \cite{CM07} were used.  In a forthcoming article we will use our method to present more efficient proofs of the results.

The indecomposable extensions composed of two \tdm s of $\VR$ were classified by Feigin and Fuchs \cite{FF80}, and projectively split uniserial extensions of arbitrary length were analyzed by O'Dell (another grandstudent of Raja's) \cite{O'D18}.  The annihilators of the simplest class of projectively split extensions were studied by Kenefake (also a grandstudent of Raja's) \cite{Ke19}.  He was not able to describe them fully, but he did describe their intersections with the subalgebra of $\dU(\VR)$ spanned by \lwv s.  We expect that complete descriptions of these annihilators, and perhaps also of the annihilators of more complicated extensions of length~$2$, may be obtained using the multidegree method.  We plan to work with Kenefake to address these questions in a future article.

In order to place our results in some context, recall that by a well-known theorem of Duflo, the annihilators of the Verma modules of a finite dimensional semisimple Lie algebra are generated by their intersection with the center of its \uea\ \cite{Du71}.  The \uea s of vector field algebras such as $\VRm$ and $\K_{(2m+1) | n}$ have no center, so they cannot admit exact analogs of Duflo's result.  However, each of them contains a distinguished finite dimensional semisimple subalgebra, its \textit{projective subalgebra}, and the \tfm s of the parent vector field algebra restrict to relative co-Verma modules of the projective subalgebra, induced from a specific maximal parabolic subalgebra.

For $\VR$ and $\Koo$, the projective subalgebras are $\dsl_2$ and $\osp_{1|2}$, respectively: the rank~$1$ cases.  Here all \tfm s are \tdm s, and they restrict to co-Verma modules of the projective subalgebra.  Therefore, their annihilators over the projective subalgebra are generated by shifts of the Casimir element (except in the case of the self-opposite-dual Verma module of $\osp_{1|2}$; its annihilator is generated by the ``ghost'', the square root of a shift of the Casimir element \cite{Pi90}).  These shifted Casimir elements play roles in the study of the full annihilators over the vector field algebra, although they do not generate them.

This article is organized as follows.  In Section~\ref{Notation} we establish general notation, and in Section~\ref{Single} the single degree method is given.  In Sections~\ref{VecR} and~\ref{Tensors} we review $\VR$, its universal enveloping algebra, and its modules, and in Sections~\ref{U2} and~\ref{F_0 F_Lambda} the single degree method is applied to the two modules of $\VR$ discussed above.  In Section~\ref{Multiple} we give the multidegree method, and in Sections~\ref{U3} and~\ref{F_lambda} it is applied to the individual \tdm s of $\VR$.  Section~\ref{Remarks} contains remarks on further applications, and Section~\ref{Related} discusses relations to work of Sierra and Walton \cite{SW14, SW16}, Billig and Futorny \cite{BF16}, and Petukhov and Sierra \cite{PS20, SP23}.

\section{Notation} \label{Notation}
Throughout the article we work over $\bC$.  Given any vector space $W$, let $\S(W)$ be its symmetric algebra, and let $\S^k(W)$ be the $k^\thup$ homogeneous component of $\S(W)$.  If $W'$ is a subspace of $W$, we write $\proj_{W'}$ for the algebra epimorphism induced by the canonical projection from $W$ to $W/W'$:
\begin{equation*}
   \proj_{W'}: \S(W) \twoheadrightarrow \S(W/W').
\end{equation*}

Let $\dg$ be a Lie algebra, with \uea\ $\dU(\dg)$.  Write $\dU_k(\dg)$ for the subspace of $\dU(\dg)$ of elements of degree~$\le k$, and
\begin{equation*}
   \proj_k: \dU_k(\dg) \twoheadrightarrow \S^k(\dg)
\end{equation*}
for the canonical projection with kernel $\dU_{k-1}(\dg)$.  The graded commutativity of $\dU(\dg)$ is the following homomorphic property: for $\Theta_1 \in \dU_{k_1}(\dg)$ and $\Theta_2 \in \dU_{k_2}(\dg)$,
\begin{equation*}
   \proj_{k_1}(\Theta_1) \proj_{k_2}(\Theta_2) = \proj_{k_1 + k_2}(\Theta_1 \Theta_2).
\end{equation*}

The \textit{symmetrizer map} in degree~$k$ is
\begin{equation*}
   \sym_k: \S^k(\dg) \to \dU_k(\dg), \qquad
   \sym_k(X_1 \cdots X_k) := {\ts\frac{1}{k!}}
   \sum_{\sigma \in S_n} X_{\sigma(1)} \cdots X_{\sigma(n)}.
\end{equation*}
It is a $\dg$-map and a right-inverse of $\proj_k$.  It follows that the full symmetrizer map $\sym$ is a degree-preserving $\dg$-isomorphism:
\begin{equation*}
   \sym := \bigoplus_k \sym_k: \S(\dg) \to \dU(\dg).
\end{equation*}

\begin{defn}
Given any subspace $I$ of $\dU(\dg)$, let $\la I \ra$ be the ideal it generates:
\begin{equation*}
   \la I \ra := \dU(\dg) I \vtsp \dU(\dg).
\end{equation*}
\end{defn}

\begin{defn}
Let $\I$ be an ideal in $\dU(\dg)$.  We say that $\I$ is generated in degree~$d$ if
\begin{itemize}
\item
$\I \cap \dU_{d-1}(\dg) = 0$;
\smallbreak \item
$\I \cap \dU_d(\dg)$ generates $\I$.
\end{itemize}

More generally, we say that $\I$ is generated in degrees $d_1 < \cdots < d_r$ if
\begin{itemize}
\item
$\I \cap \dU_{d_1-1}(\dg) = 0$;
\smallbreak \item
for $1 < s \le r$, $\big\la \I \cap \dU_{d_{s-1}}(\dg) \big\ra$ contains $\I \cap \dU_{d_s -1}(\dg)$, but not $\I \cap \dU_{d_s}(\dg)$;
\smallbreak \item
$\I \cap \dU_{d_r}(\dg)$ generates $\I$.
\end{itemize}
\end{defn}

\begin{defn}
The \textit{augmentation ideal} $\dU^+(\dg)$ is the annihilator of the trivial module:
\begin{equation*}
   \dU^+(\dg) := \Ann_\dg(\bC) = \la \dg \ra
   = \sym \big( \dg \oplus \S^2(\dg) \oplus \S^3(\dg) \oplus \cdots \big).
\end{equation*}
\end{defn}

\begin{rem}
Suppose that $I$ is an adjoint-invariant subspace of $\dU(\dg)$.  Then the left, right, and two-sided ideals which it generates are all the same.  However, it may nevertheless happen that the minimal degree occurring in $\la I \ra$ is lower than that occurring in $I$.  For example, if $\dg$ is simple, the image of $\sym_2$ generates $\dU^+(\dg)$.
\end{rem}

Given any subspace $U$ of $\dU(\dg)$, we sometimes write $U_k$ for $U \cap \dU_k(\dg)$ and $\o U_k$ for $\proj_k(U_k)$.  We shall remind the reader of this convention when using it.

\section{Annihilators generated in a single degree} \label{Single}

We now give the basic version of our method for describing annihilators.  As noted in the introduction, it applies only under a rather restrictive set of conditions.

\begin{prop} \label{machine}
Let $\pi$ be a \r\ of $\dg$ on a space $V$ such that there exist
\begin{itemize}
\item
a positive integer~$d$;
\smallbreak \item
a subspace $J$ of $\dg$;
\smallbreak \item
a subspace $\J$ of $\dU(\dg)$;
\smallbreak \item
a subspace $I$ of $\Ann_\dg(V) \cap \dU_d(\dg)$;
\end{itemize}
satisfying the following conditions:
\begin{enumerate}
\item[(a)]
$\pi: \dU(\dg) \to \End(V)$ is injective on $\J$;
\smallbreak \item[(b)]
$\dU_{d-1}(\dg) \subset \J$;
\smallbreak \item[(c)]
for $k \ge d$, $\proj_k$ maps $\J \cap \dU_k(\dg)$ onto $\S^{k-d+1}(J) \S^{d-1}(\dg)$;
\smallbreak \item[(d)]
$\proj_J \circ \proj_d$ maps $I$ onto $\S^d(\dg/J)$.
\end{enumerate}
Then
\begin{enumerate}
\item[(i)]
$\Ann_\dg(V)$ is generated in degree~$d$.
\smallbreak \item[(ii)]
$\Ann_\dg(V) \cap \dU_d(\dg) = I$.
\smallbreak \item[(iii)]
$\Ann_\dg(V) = \la I \ra$.
\smallbreak \item[(iv)]
$\J$ is a cross-section: $\dU(\dg) = \Ann_\dg(V) \oplus \J$.
\end{enumerate}
\end{prop}

\begin{proof}
During this proof, we use the following abbreviations:
\begin{equation*}
   \dU_k := \dU_k(\dg), \quad
   \I := \Ann_\dg(V), \quad
   \I_k := \I \cap \dU_k, \quad
   \J_k := \J \cap \dU_k.
\end{equation*}
Applying conditions (a) and~(b), we find that
\begin{equation*}
   \I \cap \J = 0, \qquad
   \I_{d-1} = 0.
\end{equation*}

We first prove that $\dU_d = \I_d \oplus \J_d$ and $I = \I_d$.  The sum is direct, and $I \subseteq \I_d$ by assumption.  Hence by~(b) it suffices to prove that
\begin{equation} \label{P1dEq1}
   \S^d(\dg) = \proj_d(I) + \proj_d(\J_d).
\end{equation}
By (c) and~(d), $\proj_J: \S^d(\dg) \to \S^d(\dg/J)$ has kernel $\proj_d(\J_d)$ and is surjective on $\proj_d(I)$.  This proves the claim.

The next step is to prove that for $k > d$,
\begin{equation} \label{P1dEq2}
   \dU_k = I \vtsp \dU_{k-d} \oplus \J_k.
\end{equation}
Because $I \vtsp \dU_{k-d} \subset \I_k$ and $\I_k \cap \J_k = 0$, we only need $\dU_k = I \vtsp \dU_{k-d} + \J_k$.  Assuming this for $k-1$, applying $\proj_k$ and condition~(c) bring us down to proving
\begin{equation*}
   \S^k(\dg) = \proj_d(I) \S^{k-d}(\dg) + \S^{k-d+1}(J) \S^{d-1}(\dg).
\end{equation*}

The following additional abbreviations clarify the computations:
\begin{equation*}
\dg^\ell := \S^\ell(\dg), \qquad
J^\ell := \S^\ell(J), \qquad
\o I := \proj_d(I).
\end{equation*}
The right side of the equation, $\o I \vtsp \dg^{k-d} + J^{k-d+1} \dg^{d-1}$, may be rewritten as
\begin{equation*}
   \o I (\dg^{k-d} + J \dg^{k-d-1}) + J^{k-d+1} \dg^{d-1}
   = \o I \vtsp \dg^{k-d} + J (\o I \vtsp \dg^{k-d-1} + J^{k-d} \dg^{d-1}).
\end{equation*}
By induction, the term in parentheses is $\dg^{k-1}$, so the expression becomes
\begin{equation*}
   \o I \vtsp \dg^{k-d} + J \dg^{k-1} = (\o I + J \dg^{d-1}) \dg^{k-d}.
\end{equation*}
Here the term in parentheses is $\dg^d$, by~\eqref{P1dEq1} and condition~(c).  This completes the proof of~\eqref{P1dEq2}.  The remainder of the proposition follows easily.
\end{proof}

The following auxiliary lemma will be convenient:

\begin{lemma} \label{consequences}
Assume that conditions (a)-(c) of Proposition~\ref{machine} hold.  Then
\begin{enumerate}
\item[(i)]
$\J \cap \dU_d(\dg) = J \dU_{d-1}(\dg) + \dU_{d-1}(\dg)$.
\smallbreak \item[(ii)]
$\proj_d(I) \cap J \S^{d-1}(\dg) = 0$.
\smallbreak \item[(iii)]
$\proj_J \circ \proj_d$ injects $I$ to $\S^d(\dg/J)$.
\end{enumerate}
\end{lemma}

\begin{proof}
Conditions~(b) and~(c) imply~(i).  For~(ii), suppose that $\o\Theta$ is in $\proj_d(I) \cap J \dg^{d-1}$.  By~(c), $J \dg^{d-1}$ is $\proj_d(\J_d)$, so there are $\Theta_I$ in $I$ and $\Theta_J$ in $\J_d$ both mapping to $\o\Theta$ under $\proj_d$.  The difference $\Theta_I - \Theta_J$ must be in $\dU_{d-1}(\dg)$, which is in $\J_d$.  Since~(a) implies $I \cap \J_d = 0$, we find that $\Theta_I$ and hence $\o\Theta$ are~$0$.

For~(iii), if $\proj_J \circ \proj_d(\Theta) = 0$ for some $\Theta \in I$, then $\proj_d(\Theta)$ is in $J \dg^{d-1}$.  Therefore it is~$0$ by~(ii), so $\Theta$ is in $\dU_{d-1}(\dg)$.  As before, this implies $\Theta = 0$.
\end{proof}

\section{$\VR$ and the tensor density modules} \label{VecR}

In this section we review $\VR$ and its \irr\ admissible \r s; see for example \cite{CM07, Ma92} and their references.  We write $\bN$ for $\bZ_{\ge 0}$.

\begin{defn}
The Lie algebra $\VR$ of polynomial \vf s on the line and its monomial \vf s $e_n$ are
\begin{equation*}
   \VR := \bC[x] \px = \spanrm_\bC \big\{ e_n: n \in -1 + \bN \big\},
   \qquad e_n := x^{n+1} \px.
\end{equation*}
The projective, affine, and constant subalgebras are, respectively,
\begin{equation} \label{abc}
   \da := \spanrm \big\{ e_{-1}, e_0, e_1 \big\}, \qquad
   \db := \spanrm \big\{ e_{-1}, e_0 \big\}, \qquad
   \dc := \bC e_{-1}.
\end{equation}
\end{defn}

Here $\da$ is isomorphic to $\dsl_2$, $\db$ is a Borel subalgebra, and $\dc$ is its nilradical.  The corresponding Cartan subalgebra is spanned by $e_0$, the Euler operator of $\VR$.

\begin{defn}
In any module of $\VR$, the eigenspaces, eigenvectors, and eigenvalues of $e_0$ are, respectively, the \textit{weight spaces}, \textit{weight vectors}, and \textit{weights}.  Weight vectors annihilated by $e_{-1}$ are \textit{\lwv s}.  Given a module $V$, its $\mu$-weight space, the subspace of $\mu$-\lwv s, and the subspace spanned by all \lwv s are denoted, respectively, by
\begin{equation*}
   V_\mu, \qquad V_\mu^{e_{-1}}, \qquad V^{e_{-1}}.
\end{equation*}
A module is said to be \textit{admissible} if it is spanned by its weight spaces, and all of them are finite dimensional.
\end{defn}

The defining module of $\VR$ is $\bC[x]$ with the natural action: a \vf\ $f(x) \px$ maps a polynomial $g(x)$ to $f g'$.  This action extends to the Laurent polynomials $\bC[x, x^{-1}]$, and further to the formal space $x^a \bC[x, x^{-1}]$, for any $a \in \bC/\bZ$.  Deforming it gives the \tdm s:

\begin{defn}
For $\lambda \in \bC$ and $a \in \bC / \bZ$, the \textit{\tdm} $\F_{a, \lambda}$ of $\VR$ and its action $\pi_\lambda$ are given by
\begin{equation*}
   \F_{a, \lambda} := dx^\lambda x^a \bC[x, x^{-1}], \qquad
   \pi_\lambda(f \px) (dx^\lambda g) := dx^\lambda (f g' + \lambda f' g).
\end{equation*}
\end{defn}

Here $dx^\lambda$ is a place-holder symbol of geometric origin.  Note that $dx^\lambda x^\mu$ is a weight vector of weight $\lambda + \mu$.  It is a \lwv\ \iff\ $\mu = 0$.

For $a$ nonintegral, $\F_{a, \lambda}$ is \irr.  For $a \equiv 0$, the polynomials span a submodule, denoted by $\F_\lambda$.  The corresponding quotient, which we will denote by $\F^-_\lambda$, may be thought of as being spanned by the negative monomials:
\begin{equation*}
   \F_\lambda := dx^\lambda \bC[x], \qquad
   \F^-_\lambda := \F_{\lambda, \lambda} / \F_\lambda = dx^\lambda \bC[x^{-1}] x^{-1}.
\end{equation*}
These modules are both \irr, excepting $\F_0$ and $\F^-_1$, which each have exactly one submodule: $\bC$ and $dx \vtsp \bC[x^{-1}] x^{-2}$, respectively.  The de~Rham differential $dx \px$ gives isomorphisms from the quotient $\F_0/\bC$ to $\F_1$ and from $\F^-_0$ to $dx \vtsp \bC[x^{-1}] x^{-2}$.

The quotient of $\F^-_1$ by $dx \vtsp \bC[x^{-1}] x^{-2}$ is trivial and is given by the residue map, the unique $\VR$-covariant surjection from any \tdm\ to the trivial module.  It extends to all of $\F_1$:

\begin{defn}
The \textit{residue map} $\res: \F_1 \to \bC$ is $\res \big(dx \vtsp g(x) \big) := \frac{1}{2 \pi i} \oint_0 dx \vtsp g(x)$.
\end{defn}

There is a natural multiplication map on the \tdm s:
\begin{equation*}
   \mult: \F_{a, \lambda} \ot \F_{a', \lambda'} \to \F_{a+a', \lambda+\lambda'}, \qquad
   dx^\lambda x^\mu \ot dx^{\lambda'} x^{\mu'} \mapsto dx^{\lambda+\lambda'} x^{\mu+\mu'}.
\end{equation*}
It is $\VR$-covariant, and is known as the \textit{zeroth transvectant}.  In conjunction with the residue map, it defines a family of nondegenerate invariant bilinear forms:
\begin{equation*}
   \res \circ \mult: \F_{a, \lambda} \ot \F_{-a, 1-\lambda} \to \bC.
\end{equation*}
Thus we have the following identifications of restricted duals:
\begin{equation} \label{dual}
   \F_{a, \lambda}^* \cong \F_{-a, 1-\lambda}, \qquad
   \F_\lambda^* \cong \F^-_{1-\lambda}.
\end{equation}

Mathieu obtained the following result en route to his proof of the Kac conjecture:

\begin{thm}[\cite{Ma92}]
The trivial module and the \tdm s exhaust the \irr\ admissible modules of $\VR$.
\end{thm}

\section{The structure of $\dU(\VR)$} \label{Tensors}

Here we review $\dU(\VR)$ and its action on the $\F_\lambda$.

\begin{lemma} \label{VRF}
The adjoint \r\ of $\VR$ on itself is isomorphic to $\F_{-1}$, via the map $f(x) \px \mapsto dx^{-1} f(x)$.
\end{lemma}

Thus the symmetrizer map shows that the adjoint action on $\dU(\VR)$ is isomorphic to the action on the symmetric algebra $\S(\F_{-1})$.

\begin{lemma} \label{Tensor F}
For any scalars $\lambda_1, \ldots, \lambda_k$, there is a $\db$-module isomorphism
\begin{equation*}
   \bigotimes_{j=1}^k \F_{\lambda_j} \bcong
   \bigoplus_{n=0}^\infty t_k(n) \F_{n + \lambda_1 + \cdots + \lambda_k},
   \qquad t_k(n) := \binom{n + k - 2}{k - 2}.
\end{equation*}
\end{lemma}

\begin{proof}
First use the fact that $e_{-1}$ acts surjectively on $\F_\lambda$ to see that it also acts surjectively on $\bigotimes_j \F_{\lambda_j}$.  Then check that any $\db$-module on which $e_{-1}$ acts surjectively and whose weights are bounded below is a direct sum of \tdm s.

Standard counting arguments show that the dimension of the $\big( n + \sum_j \lambda_j \big)$-weight space of $\bigotimes_j \F_{\lambda_j}$ is $w_k(n) := \binom{n + k - 1}{k - 1}$.  By the $e_{-1}$-surjectivity, the dimension of the subspace of \lwv s is $w_k(n) - w_k(n-1)$, which is $t_k(n)$.
\end{proof}

\begin{lemma} \label{Sym F}
For any scalar $\lambda$, there is a $\db$-module isomorphism
\begin{equation*}
   \S^k (\F_\lambda) \bcong
   \bigoplus_{n=0}^\infty q_k(n) \F_{n + k \lambda},
\end{equation*}
where $q_k(n)$ denotes the number of partitions of $n$ with all part sizes in $[2, k]$.
\end{lemma}

\begin{proof}
A Young diagram conjugation argument shows that the dimension of the $(n + k \lambda)$-weight space of $\S^k (\F_\lambda)$ is $p_k(n)$, the number of partitions of $n$ with all part sizes in $[1, k]$.  As before, $e_{-1}$ acts surjectively on $\S^k (\F_\lambda)$, and so the dimension of the subspace of \lwv s is $p_k(n) - p_k(n-1)$, which is $q_k(n)$.
\end{proof}

We emphasize that these are only $\db$-decompositions, and they are not unique.  The precise $\da$-decompositions may be found in the appendix of \cite{CM07}.  Under $\VR$ these modules are in general indecomposable.

Observe that the tensor density action $\pi_\lambda$ carries $\dU(\VR)$ into the Weyl algebra $\bC[x, \px]$.  This restricts the possible actions of lowest weight elements.  It also permits us to reduce the study of the annihilators of the $\F_{a, \lambda}$ to those of the $\F_\lambda$.

\begin{lemma} \label{LWV action}
Let $\Theta$ be an element of $\dU(\VR)^{e_{-1}}_n$.  For all $\lambda \in \bC$,
\begin{enumerate}
\item[(i)]
If $n > 0$, then $\pi_\lambda(\Theta) = 0$.
\smallbreak \item[(ii)]
If $n \le 0$, then $\pi_\lambda(\Theta)$ is a multiple of $\px^{|n|}$.
\end{enumerate}
\end{lemma}

\begin{proof}
The commutant of $\px$ in $\bC[x, \px]$ is $\bC[\px]$.  Because $\pi_\lambda(e_{-1}) = \px$ and $\Theta$ is a lowest weight element of weight~$n$, $\pi_\lambda(\Theta)$ must lie in $\bC[\px]_n$.
\end{proof}

\begin{lemma} \label{F action}
Suppose that $U$ is a $\db$-invariant subspace of $\dU(\VR)$ which is $\db$-isomorphic to $\F_n$, for some $n > 0$.  Then $\pi_\lambda(U) = 0$ for all $\lambda$.
\end{lemma}

\begin{proof}
Suppose that $\pi_\lambda(\Theta) \not= 0$ for some $\Theta \in U$ of weight $n + m$, where $m \in \bN$.  Note that $\ad(e_{-1})^{m+1} \Theta = 0$.  Let $m_0$ be maximal for $\pi_\lambda(\ad(e_{-1})^{m_0} \Theta) \not= 0$.  Deduce that $\pi_\lambda(\ad(e_{-1})^{m_0} \Theta)$ is an element of $\bC[\px]$ of positive weight, a contradiction.
\end{proof}

\begin{prop} \label{no a}
For all~$a$, the annihilator of $\F_{a, \lambda}$ is equal to that of $\F_\lambda$:
\begin{equation*}
   \Ann_{\VR}(\F_{a, \lambda}) = \Ann_{\VR}(\F_\lambda) = \Ann_{\VR}(\F^-_\lambda).
\end{equation*}
\end{prop}

\begin{proof}
Recall that as spaces, $\F_\lambda$ is $\bC[x]$, $\F^-_\lambda$ is $x^{-1} \bC[x^{-1}]$, and $\F_{a, \lambda}$ is $x^a \bC[x^{\pm1}]$; $dx^\lambda$ is just a place-holder indicating the action of $\dU(\VR)$, which acts by polynomial \dog s in all cases.  Any polynomial \dog\ which annihilates any one of these spaces must annihilate all of them.
\end{proof}

\section{The structure of $\dU_2(\VR)$} \label{U2}

In order to analyze annihilators generated in degree~$2$, we must recall the structure of $\dU_2(\VR)$.  The degree~$2$ symmetrizer map, $\sym_2$, is a $\VR$-injection from $\S^2(\VR)$ to $\dU_2(\VR)$.  It will be helpful to abbreviate its image as $\dU^2$:
\begin{equation} \label{U^2}
   \dU_2(\VR) = \bC \oplus \VR \oplus \dU^2,
   \qquad \dU^2 := \sym_2 \big( \S^2(\VR) \big).
\end{equation}

By Lemma~\ref{VRF}, $\dU^2$ is $\VR$-isomorphic to $\S^2(\F_{-1})$, and by Lemma~\ref{Sym F}, there exists a (non-unique) $\db$-isomorphism
\begin{equation} \label{S2b}
   \S^2 (\F_{-1}) \bcong \F_{-2} \oplus \F_0 \oplus \F_2 \oplus \F_4 \oplus \cdots.
\end{equation}
In particular, for each weight in $-2 + 2\bN$, the lowest weight space of $\dU^2$ is $1$-dimensional.  In weight~$-2$, it is clearly $\bC e_{-1}^2$, and there is a well-known technique for constructing it in other weights: the action of the \textit{step algebra} element
\begin{equation*}
   S := 2 e_2 (2 e_0 - 1) - 3 e_1^2
\end{equation*}
of $\dU(\VR)$ on any module maps any \lwv\ of weight $\mu$ either to zero or to a \lwv\ of weight $\mu + 2$.  Keeping track only of the coefficient of $e_{2\ell+1} e_{-1}$, one obtains $\ad(S)^{\ell+1} e_{-1}^2 \not= 0$ for all~$\ell$.  This proves:

\begin{lemma} \label{S formula}
For $\ell \in -1+\bN$, $\ad(S)^{\ell+1} e_{-1}^2$ is a lowest weight element of $\dU^2$ of weight $2\ell$.  These elements form a basis of the lowest weight space $(\dU^2)^{e_{-1}}$.
\end{lemma}

The lowest weight space of $\dU^2$ of weight~$0$ is spanned by the \textit{Casimir element} of $\da$, the generator of the center of $\dU(\da)$:
\begin{equation*}
   Q := e_0^2 - e_0 - e_1 e_{-1} = -{\ts\frac{1}{48}} \ad(S) e_{-1}^2.
\end{equation*}
In any $\da$-module, this element acts by the scalar $\lambda^2 - \lambda$ on any \lwv\ of weight~$\lambda$.  This can be used to see that it acts by this same scalar on the \tdm s $\F_{a, \lambda}$, for all~$a$.  We write $q(\lambda)$ for said scalar:
\begin{equation} \label{Qq}
   \pi_\lambda(Q) = q(\lambda) := \lambda^2 - \lambda.
\end{equation}
Observe that $q(\lambda') = q(\lambda)$ \iff\ $\lambda'$ is $\lambda$ or $1 - \lambda$.  

We will also encounter the lowest weight space of $\dU^2$ of weight~$2$.  It is spanned by $\ad(e_2)Q$, which we abbreviate as $Q^{e_2}$:
\begin{equation} \label{Qe2}
   Q^{e_2} = 3e_1^2 - 2e_2 (2e_0 + 1) + e_3 e_{-1} = -{\ts\frac{1}{96}} \ad(S)^2 e_{-1}^2.
\end{equation}

The Casimir element gives a refinement of \eqref{S2b}:

\begin{lemma} \label{S2a}
There is a unique $\da$-decomposition
\begin{equation*}
   \S^2(\F_{-1}) \acong \F_{-2} \oplus \F_0 \oplus \F_2 \oplus \F_4 \oplus \cdots.
\end{equation*}
\end{lemma}

\begin{proof}
By~\eqref{S2b}, there exist such decompositions under~$\db$.  Since the scalars $q(\lambda)$ are distinct for $\lambda \in -2 + 2\bN$, elementary arguments show that one of these $\db$-decompositions is given by the eigenspaces of~$Q$, which are $\da$-invariant.  Any $\da$-module which is $\db$-isomorphic to $\F_\lambda$ is in fact $\da$-isomorphic to $\F_\lambda$.
\end{proof}

We remark that the explicit $\da$-maps $\S^2(\F_{-1}) \to \F_{2\ell}$ are instances of \textit {symmetric transvectants}.  The next proposition amalgamates the preceding results:

\begin{prop} \label{U2a}
There is a unique $\da$-decomposition
\begin{equation*}
   \dU^2 = G_{-2} \oplus G_0 \oplus G_2 \oplus G_4 \oplus \cdots,
\end{equation*}
where $G_{2\ell}$ is an $\da$-submodule of $\dU^2$ which is $\da$-isomorphic to $\F_{2\ell}$.  The lowest weight lines are spanned by $\ad(S)^{\ell+1} e_{-1}^2$, or equivalently, by $e_{-1}^2, Q, Q^{e_2}, \ldots$.
\end{prop}

\begin{defn}
For $\ell \in -1+\bN$, set $H_{2\ell} := G_{2\ell} \oplus G_{2\ell+2} \oplus \cdots$.
\end{defn}

\begin{prop} \label{H2ell}
\begin{enumerate}
\item[(i)]
For $\ell \in -1+\bN$, $H_{2\ell}$ is a $\VR$-submodule of $\dU^2$.
\smallbreak \item[(ii)]
The quotient $H_{2\ell} / H_{2\ell+2}$ is $\VR$-isomorphic to $\F_{2\ell}$.
\smallbreak \item[(iii)]
Under $\VR$, $\dU^2$ is uniserial.  Its only submodules are the $H_{2\ell}$ and $\bC Q \oplus H_2$.
\end{enumerate}
\end{prop}

\begin{proof}
By the Poincar\'e-Birkhoff-Witt theorem with $\da$ on the right, the $\VR$-submodule of $\dU^2$ generated by $H_{2\ell}$ has weights in $2\ell + \bN$.  This implies~(i).  For~(ii), elementary arguments show that any $\VR$-module which is $\da$-isomorphic to $\F_\lambda$ is in fact $\VR$-isomorphic to $\F_\lambda$.

For~(iii), suppose that $U$ is a $\VR$-submodule of $\dU^2$.  Then its minimal weight is $2\ell$ for some $\ell \in -1+\bN$, and it contains the lowest weight element $\ad(S)^{\ell+1} e_{-1}^2$.  Consequently, it contains $\ad(S)^{m+1} e_{-1}^2$ for $m \ge \ell$.  Since $\F_\mu$ is $\da$-\irr\ for $\mu \not\in -\oh \bN$, we find that if $\ell > 0$, then $U = H_{2\ell}$.

If $\ell = 0$, then $U$ contains $Q$.  Because $\F_0$ has exactly one proper non-trivial $\VR$-submodule, $\bC$, there are two possibilities: $U$ could be either $H_0$ or $\bC Q \oplus H_2$.  Note that they are distinguished by the dimension of the $1$-weight space $U_1$.

If $\ell = -1$, then $U$ contains $e_{-1}^2$.  Since $\dU^2 / H_0 \cong \F_{-2}$ is $\VR$-\irr, $U$ must surject to it.  Because $\ad(e_3) e_{-1}^2$ and $\ad(e_2 e_1) e_{-1}^2$ are not proportional, $U_1$ is $2$-dimensional.  It follows that $U = \dU^2$.  The uniseriality is now clear.
\end{proof}

\begin{cor} \label{ZQS}
\begin{enumerate}
\item[(i)]
$Q$ has a unique $\ad(e_{-1})$-preimage $Z$ in $\dU^2$:
\begin{equation*}
   Z := {\ts\frac{1}{2}} (e_1 e_0 - e_2 e_{-1} - e_1).
\end{equation*}
\smallbreak \item[(ii)]
$H_0$ is generated as a $\VR$-module by $Z$.
\smallbreak \item[(iii)]
$\bC Q \oplus H_2$ is generated as a $\VR$-module by $Q$.
\smallbreak \item[(iv)]
For $\ell \not= 0$, $H_{2\ell}$ is generated as a $\VR$-module by $\ad(S)^{m+1} e_{-1}^2$.
\smallbreak \item[(v)]
In particular, $H_2$ is generated as a $\VR$-module by $Q^{e_2}$.
\end{enumerate}
\end{cor}

\begin{proof}
For~(i), verify that $\ad(e_{-1}) Z = Q$ and use the fact that $\dU_2(\VR)$ contains no $1$-lowest weight elements.  Everything else follows from Proposition~\ref{H2ell}.
\end{proof}

\begin{rem}
In fact, $Z$ is the only preimage of $Q$ in $\dU_2(\VR)$.  It projects to a lowest weight element of the $\VR$-module $H_0 / (\bC Q \oplus H_2)$, which is isomorphic to $\F_1$.
\end{rem}

\begin{lemma} \label{ZQe2}
For $\lambda \in \bC$, $H_2$, $Z$, and $H_0$ have the following images under $\pi_\lambda$:
\begin{equation*}
   \pi_\lambda(H_2) = 0, \qquad \pi_\lambda(Z) = q(\lambda) x,
   \qquad \pi_\lambda(H_0) = \pi_\lambda(G_0) = q(\lambda) \bC[x].
\end{equation*}
\end{lemma}

\begin{proof}
For $H_2$, use Lemma~\ref{F action}.  For $Z$, combine \eqref{Qq} with the fact that $\big( \bC[x, \px] \big)_1^{e_{-1}} = 0$.  For~$H_0$, act by $e_1$ repeatedly on the $Z$-action formula.
\end{proof}

\section{Annihilators over $\VR$ generated in degree~2} \label{F_0 F_Lambda}

In this section we use Proposition~\ref{machine} to give clarified proofs of two of the results of \cite{CM07}.  As discussed in the introduction, the approach taken here is adaptable to more general vector field Lie algebras and superalgebras.

We begin with some elementary results on duals.  In the context of admissible \r s, they apply also to restricted duals.  We omit the proofs.

\begin{defn}
Let $\dg$ be a Lie algebra.  The transpose anti-involution of $\dU(\dg)$ is
\begin{equation*}
   \Theta \mapsto \Theta^T, \qquad (X_1 \cdots X_k)^T := (-1)^k X_k \cdots X_1.
\end{equation*}
\end{defn}

\begin{lemma} \label{transpose parity}
Transposition is $\ad$-invariant and acts by $(-1)^k$ on $\sym_k(\S^k(\dg))$.
\end{lemma}

\begin{lemma} \label{transpose}
For any \r\ $(\pi, V)$ of a Lie algebra $\dg$, the annihilator of its dual is the transpose of its annihilator:
\begin{equation*}
   \Ann_\dg(V^*) = \Ann_\dg(V)^T.
\end{equation*}
\end{lemma}

\begin{thm}[\cite{CM07}] \label{I0}
$\Ann_{\VR}(\F_1) = \Ann_{\VR}(\F_0) = \la Z \ra$.
\end{thm}

\begin{proof}
For the first equality, combining \eqref{dual}, Proposition~\ref{no a}, and Lemma~\ref{transpose} shows that $\Ann_{\VR}(\F_1)$ is the transpose of $\Ann_{\VR}(\F_0)$.  Since $Z \in \dU^2$, it is transpose-invariant by Lemma~\ref{transpose parity}, so it suffices to prove the second equality.  

By Lemma~\ref{ZQe2}, $Z$ annihilates $\F_0$.  Hence Corollary~\ref{ZQS}(ii) implies that $H_0$ annihilates $\F_0$.  Recall the constant subalgebra $\dc$ from~\eqref{abc}.  The ingredients for Proposition~\ref{machine} are $d := 2$, $I := H_0$, $J := \dc$, and
\begin{equation*}
   \J := \dU_1(\VR) \dU(\dc) = \spanrm \big\{ e_n e_{-1}^k, e_{-1}^k: n, k \in \bN \big\}.
\end{equation*}

We must check that conditions~(a)-(d) of Proposition~\ref{machine} hold.  For~(a), note that the monomials $e_n e_{-1}^k$ act by the monomial basis elements $x^{n+1} \px^{k+1}$ of $\bC[x, \px]$.  Conditions~(b) and~(c) are clear.

For~(d), by Lemma~\ref{consequences}(iii) it suffices to show that $H_0$ and $\S^2((\VR)/\dc)$ have the same weight space dimensions.  By definition, $H_0$ is $\da$-isomorphic to $\oplus_0^\infty \F_{2\ell}$.  By Lemma~\ref{Sym F}, $\S^2(\F_0)$ is $\db$-isomorphic to $\oplus_0^\infty \F_{2\ell}$ (in fact, $\da$-isomorphic, by Casimir eigenvalues).  To complete the proof, note that $(\VR)/\dc$ is $\db$-isomorphic to $\F_0$.
\end{proof}

The following definition will be useful both here and subsequently.

\begin{defn}
Let $\Lambda$ be a central indeterminate.  The \textit{universal \tdm} $\F_\Lambda$ of $\VR$ and its action $\pi_\Lambda$ are given by
\begin{equation*}
   \F_\Lambda := dx^\Lambda \bC[\Lambda, x], \qquad
   \pi_\Lambda(f \px) (dx^\Lambda g):= dx^\Lambda (fg' + \Lambda f'g).
\end{equation*}
\end{defn}

\begin{thm}[\cite{CM07}] \label{I2}
$\bigcap_{\lambda \in \bC} \Ann_{\VR}(\F_\lambda) = \la Q^{e_2} \ra$.
\end{thm}

\begin{proof}
Clearly $\bigcap_{\lambda \in \bC} \Ann_{\VR}(\F_\lambda)$ is the annihilator of both $\bigoplus_\lambda \F_\lambda$ and $\F_\Lambda$.  By Lemma~\ref{ZQe2}, it contains $H_2$.  The ingredients for Proposition~\ref{machine} applied to $\F_\Lambda$ are $d := 2$, $I := H_2$, $J := \db$, and
\begin{equation*}
   \J := \dU_1(\VR) \dU(\db) =
   \spanrm \big\{ e_n e_0^{k_0} e_{-1}^{k_{-1}}, e_0^{k_0} e_{-1}^{k_{-1}}:
   n \in \bZ^+,\, k_0, k_{-1} \in \bN \big\}.
\end{equation*}

We must check that conditions~(a)-(d) of Proposition~\ref{machine} hold.  For~(a), observe that $\pi_\Lambda$ maps $\dU(\VR)$ into the algebra $\bC[\Lambda, x, \px]$.  Consider the total $(\Lambda, \px)$ degree on this algebra.  Check that it defines a filtration compatible with multiplication, and that $\pi_\Lambda$ maps $\dU_k(\VR)$ to operators of total degree~$\le k$.  Consequently, to prove that $\pi_\Lambda$ is injective on $\J$, it is enough to prove that the $(\Lambda, \px)$-symbols of the $\pi_\Lambda$-images of the given basis elements of $\J$ are linearly independent.

For reference, $\pi_\Lambda(e_0^{k_0} e_{-1}^{k_{-1}})$ and $\pi_\Lambda(e_n e_0^{k_0} e_{-1}^{k_{-1}})$ have the following $(\Lambda, \px)$-symbols:
\begin{align} \label{symbols}
   & \pi_\Lambda(e_0^{k_0} e_{-1}^{k_{-1}}) \equiv
   \sum_{\ell = 0}^{k_0}
   \binom{k_0}{\ell}x^{k_0 - \ell} \Lambda^\ell \px^{k_0 + k_{-1} - \ell},
   \nonumber \\[4pt]
   & \pi_\Lambda(e_n e_0^{k_0} e_{-1}^{k_{-1}}) \equiv
   x^{n + k_0 + 1} \px^{k_0 + k_{-1} + 1}
   + (n + 1) x^n \Lambda^{k_0 + 1} \px^{k_{-1}}
   \\[4pt] \nonumber
   & \qquad + \sum_{\ell = 1}^{k_0}
   \bigg[ \binom{k_0}{\ell} + (n + 1) \binom{k_0}{\ell - 1} \bigg]
   x^{n + k_0 - \ell + 1} \Lambda^\ell \px^{k_0 + k_{-1} - \ell + 1}.
\end{align}

Consider the basis elements of $\J$ of weight~$\mu$ and degree~$k > 0$.  These are
\begin{align*}
   & e_\mu e_0^{k-1},\, e_{\mu+1} e_0^{k-2} e_{-1},\, e_{\mu+2} e_0^{k-3} e_{-1}^2,\,
   \ldots,\, e_{\mu+k-1} e_{-1}^{k-1} \ \mathrm{\ for\ }\ \mu > 0, \\[4pt]
   & e_0^{k+\mu} e_{-1}^{-\mu},\, e_1 e_0^{k-1+\mu} e_{-1}^{1-\mu},\, \ldots,\,
   e_{k-1+\mu} e_{-1}^{k-1} \ \mathrm{\ for\ }\ 0 \ge \mu \ge -k.
\end{align*}
Using \eqref{symbols}, one finds that in both lists the $\Lambda$-degree of the $(\Lambda, \px)$-symbol decreases by~$1$ with each element, from $k$ to $1$ for $\mu > 0$, and from $k+\mu$ to $1$ for $\mu \le 0$.  Therefore the symbols are independent.

Conditions~(b) and~(c) are clear.  Treat~(d) as in Theorem~\ref{I0}: by Lemma~\ref{consequences}(iii), it suffices to show that $H_2$ and $\S^2((\VR)/\db)$ have the same weight space dimensions.  By definition, $H_2$ is $\da$-isomorphic to $\oplus_1^\infty \F_{2\ell}$.  By Lemma~\ref{Sym F}, $\S^2(\F_1)$ is $\db$-isomorphic to $\oplus_1^\infty \F_{2\ell}$.  Finally, $(\VR)/\db$ is $\db$-isomorphic to $\F_1$.
\end{proof}

\section{Annihilators generated in multiple degrees} \label{Multiple}

In this section, the method given in Section~\ref{Single} is adapted to annihilators generated in more than one degree.  As before, it is applicable only under strong conditions.

\begin{prop} \label{multimachine}
Let $\pi$ be a \r\ of $\dg$ on a space $V$ such that there exist
\begin{itemize}
\item
positive integers~$d_1 < d_2 < \cdots < d_r$;
\smallbreak \item
subspaces $J_{d_1} \supset J_{d_2} \supset \cdots \supset J_{d_r}$ of $\dg$;
\smallbreak \item
a subspace $\J$ of $\dU(\dg)$;
\smallbreak \item
a subspace $I$ of $\Ann_\dg(V) \cap \dU_{d_r}(\dg)$;
\end{itemize}
satisfying the following conditions:
\begin{enumerate}
\item[(a)]
$\pi: \dU(\dg) \to \End(V)$ is injective on $\J$;
\smallbreak \item[(b)]
$\dU_{d_1 - 1}(\dg) \subset \J$;
\smallbreak \item[(c)]
for $1 \le s \le r$ and $d_s \le k < d_{s+1}$ (take $d_{r+1}$ to be infinity),
\begin{equation*}
   \proj_k \bigl( \J \cap \dU_k(\dg) \bigr) = \S^{k - d_1 + 1}(J_{d_s}) \S^{d_1 - 1}(\dg);
\end{equation*}
\smallbreak \item[(d)]
for $1 \le s \le r$,
\begin{equation*}
   \proj_{d_s} \bigl( I \cap \dU_{d_s}(\dg) \bigr)
   + \S^{d_s - d_1 + 1}(J_{d_s}) \S^{d_1 - 1}(\dg)
   = \S^{d_s}(\dg);
\end{equation*}
\smallbreak \item[(e)]
for $1 < s \le r$, $I \cap \dU_{d_s - 1}(\dg) = I \cap \dU_{d_{s-1}}(\dg)$.
\end{enumerate}
Then
\begin{enumerate}
\item[(i)]
$\Ann_\dg(V)$ is generated in some subset of the degrees~$d_1, \ldots, d_r$.
\smallbreak \item[(ii)]
For $1 < s \le r$,
\begin{equation*}
   \Ann_\dg(V) \cap \dU_{d_s - 1}(\dg)
   = \bigl( \Ann_\dg(V) \cap \dU_{d_{s-1}}(\dg) \bigr) \dU_{d_s - d_{s-1} - 1}(\dg).
\end{equation*}
\smallbreak \item[(iii)]
$\Ann_\dg(V) = \la I \ra$.
\smallbreak \item[(iv)]
$\J$ is a cross-section: $\dU(\dg) = \Ann_\dg(V) \oplus \J$.
\end{enumerate}
\end{prop}

\begin{rem}
We index the subspaces $J_\bullet$ by $d_s$ rather than simply by~$s$, as it is convenient for their subscripts to indicate the degrees in which they take effect.  Also, note that in~(c), the exponents involve $d_1$, not $d_s$.
\end{rem}

\begin{proof}
We use the same abbreviations as in the proof of Theorem~\ref{machine}: $\dU_k := \dU_k(\dg)$, and for any subspaces $J \subseteq \dg$ and $\L \subseteq \dU$,
\begin{equation} \label{abbrevs}
   J^k := \S^k(J), \qquad
   \L_k := \L \cap \dU_k, \qquad
   \o\L_k := \proj_k(\L_k).
\end{equation}

The core of the argument consists in proving $\la I \ra \oplus \J = \dU$.  By~(a), $\la I \ra \cap \J = 0$, so by induction, it suffices to prove $\o{\la I \ra}_k + \o\J_k = \dg^k$ for all~$k$.  By~(b), this holds for $k < d_1$, and by (c) and~(d), it holds for $k = d_s$.  Proceeding by induction for $d_s < k < d_{s+1}$, we come down to proving that for such $k$,
\begin{equation*}
   \o I_{d_s} \dg^{k-d_s} + J_{d_s}^{k-d_1+1} \dg^{d_1-1} = \dg^k.
\end{equation*}
Expanding $\dg^{k-d_s}$ as $\dg^{k-d_s} + J_{d_s} \dg^{k-d_s-1}$, the left side becomes
\begin{equation*}
   \o I_{d_s} \dg^{k-d_s} + J_{d_s} (\o I_{d_s} \dg^{k-d_s-1} + J_{d_s}^{k-d_1} \dg^{d_1-1}).
\end{equation*}
By induction, the space in parentheses is $\dg^{k-1}$, so the expression is
\begin{equation*}
   \dg^{k-d_s} (\o I_{d_s} + J_{d_s} \dg^{d_s - 1}).
\end{equation*}
By~(d), this is $\dg^k$.

Statements~(ii)-(iv) now follow from $I \subset \Ann_\dg(V)$, and~(i) follows from~(e).
\end{proof}

\begin{rem}
This argument does not rule out the possibility that $\Ann_\dg(V)$ is generated in a proper subset of the degrees $d_1, \ldots, d_r$.  However, in our application of the result, we will have $r = 2$, $d_1 = 2$, and $d_2 = 3$, and direct arguments will prove that the ideal is generated in both degrees.
\end{rem}

\section{The structure of $\dU_3(\VR)$} \label{U3}

In this section and the next we will write $\dU$ for $\dU(\VR)$, and we maintain the abbreviations \eqref{abbrevs}.  Building on \eqref{U^2}, we have
\begin{equation} \label{U3U^3}
   \dU_3 = \bC \oplus \VR \oplus \dU^2 \oplus \dU^3,
   \qquad \dU^3 := \sym_3 \big( \S^3(\VR) \big).
\end{equation}

By Lemma~\ref{VRF}, $\dU^3$ is $\VR$-isomorphic to $\S^3 (\F_{-1})$.  Applying Lemma~\ref{Sym F}, we find that $\S^3 (\F_{-1})$ is $\db$-isomorphic to
\begin{equation} \label{S3b}
   \bigoplus_{n_2, n_3 \in \bN} \F_{2n_2 + 3n_3 - 3}
   = \F_{-3} \oplus \F_{-1} \oplus \F_0 \oplus \F_1
   \oplus \F_2 \oplus 2\F_3 \oplus \F_4 \oplus \cdots.
\end{equation}

There is an analog of Lemma~\ref{S2a} refining this decomposition.  Let us recall the following standard results in $\dsl_2$-theory:
\begin{itemize}
\item
For $\mu \in -\oh \bN$, there is up to isomorphism a unique indecomposable $\da$-module $\t\F_\mu$ that is a $\db$-split extension of $\F_\mu$ by $\F_{1-\mu}$.
\smallbreak \item
Any $\da$-module that is $\db$-isomorphic to $\F_\mu \oplus \F_{1-\mu}$ is $\da$-isomorphic to either $\t\F_\mu$ or the $\da$-direct sum $\F_\mu \oplus \F_{1-\mu}$.
\smallbreak \item
These two $\da$-modules are distinguished by their \textit{highest weight spaces}, the kernels of the actions of $e_1$: $\t\F_\mu$ has no highest weight vectors, while the split module has a highest weight line of weight $-\mu$.
\smallbreak \item
$Q$ acts non-semisimply on $\t\F_\mu$, with unique eigenvalue $q(\mu)$.
\end{itemize}

\begin{lemma} \label{S3a}
There is a unique $\da$-decomposition
\begin{equation*}
   \S^3 (\F_{-1}) \acong \F_{-3} \oplus \F_{-1} \oplus \t\F_0
   \oplus \F_2 \oplus 2\F_3 \oplus \F_4 \oplus \cdots.
\end{equation*}
It differs from the $\db$-decomposition \eqref{S3b} only in that $\F_0 \oplus \F_1$ is replaced by $\t\F_0$.
\end{lemma}

\begin{proof}
First use \eqref{S3b} to deduce that the $Q$-generalized eigenspaces of $\S^3(\F_{-1})$ are $\db$-isomorphic to
\begin{equation} \label{S3 Q}
   (\F_{-3} \oplus F_4), \qquad (\F_{-1} \oplus F_2), \qquad (\F_0 \oplus F_1), \qquad 2\F_3,
\end{equation}
and multiples of $\F_m$ for $m \in 5 + \bN$.

It is not difficult to prove that the highest weight space of $\S(\F_\mu)$ is non-zero \iff\ $\mu \in -\oh \bN$, in which case it is the highest weight space of the symmetric algebra of the unique finite-dimensional $\da$-submodule of $\F_\mu$, the span of $\{ dx^\mu x^n: 0 \le n \le -2\mu \}$.  It follows that $\S^3(\F_{-1})$ has exactly two highest weight lines, one of weight $3$ and one of weight~$1$.  Therefore the first two generalized eigenspaces in \eqref{S3 Q} are $\da$-split, but the third is not, as there is no highest weight line of weight~$0$.

All the other generalized eigenspaces are $\da$-isomorphic to multiples of $\F_m$, because any lowest weight vector of weight $\mu \not\in -\oh \bN$ generates an $\da$-copy of $\F_\mu$.
\end{proof}

Lemma~\ref{S3a} gives the $\dU^3$-analog of Proposition~\ref{U2a}:

\begin{prop} \label{U3a}
There is a unique $\da$-decomposition
\begin{equation*}
   \dU^3 = K_{-3} \oplus K_{-1} \oplus \t K_0 \oplus K_2 \oplus K_3 \oplus K_4 \oplus \cdots,
\end{equation*}
where the summands are $\da$-submodules of $\dU^3$ such that
\begin{equation*}
   K_{-3} \acong \F_{-3}, \qquad K_{-1} \acong \F_{-1}, \qquad
   \t K_0 \acong \t\F_0, \qquad K_2 \acong \F_2,
\end{equation*}
and for $\mu \in 3 + \bN$, $K_\mu$ is $\da$-isomorphic to a multiple of $\F_\mu$.  Moreover, $\t K_0$ contains a unique $\da$-submodule $K_1$ that is $\da$-isomorphic to $\F_1$.
\end{prop}

Combining \eqref{U3U^3} with Propositions~\ref{U2a} and~\ref{U3a} gives the $\da$-structure of $\dU_3$:

\begin{prop} \label{U_3a}
$\dU_3$ has the following $\da$-decomposition:
\begin{align*}
   \dU_3 & = \bC \oplus \VR \oplus (G_{-2} \oplus G_0 \oplus G_2 \oplus \cdots)
   \oplus (K_{-3} \oplus K_{-1} \oplus \t K_0 \oplus \cdots) \\[4pt]
   & \acong \F_{-3} \oplus \F_{-2} \oplus 2\F_{-1} \oplus \bC \oplus \F_0
   \oplus \t\F_0 \oplus 2\F_2 \oplus 2\F_3 \oplus 2\F_4 \oplus 3\F_5 \oplus \cdots.
\end{align*}
\end{prop}

We will need the lowest weight elements of $\dU^3$.  Recall the element $Z$ of $\dU^2$ from Corollary~\ref{ZQS}, and define
\begin{equation*}
   Y := Q (e_0 - \oh) - Z e_{-1}, \qquad y(\lambda) := (\lambda - \oh) q(\lambda).
\end{equation*}

\begin{lemma} \label{U3stuff}
\begin{enumerate}
\item[(i)]
$\pi_\lambda(Y) = y(\lambda)$ for all $\lambda$.
\smallbreak \item[(ii)]
The lowest weight lines of $\dU^3$ of weights~$-3$, $-1$, $0$, $1$ are spanned by
\begin{equation*}
   e_{-1}^3, \qquad (Q - {\ts\frac{1}{3}}) e_{-1}, \qquad Y, \qquad Q^{e_2} e_{-1}.
\end{equation*}
\smallbreak \item[(iii)]
$K_{-1} = (Q - \frac{1}{3}) \VR$.
\smallbreak \item[(iv)]
$\ad(e_1) Y = \oh Q^{e_2} e_{-1}$.
\end{enumerate}
\end{lemma}

\begin{proof}
For~(i), apply Lemma~\ref{ZQe2}.  In~(ii), the quantities are lowest weight elements of $\dU_3$ of the correct weights: this is clear for all but $Y$, and Corollary~\ref{ZQS}(i) can be used to prove it for $Y$.  We must prove that they are actually in $\dU^3$.  This is clear for $e_{-1}^3$.  For $Y$, recall from Lemma~\ref{transpose parity} that transposition preserves lowest weight elements and acts by $(-1)^k$ on $\dU^k$.  In particular, $Q^T = Q$ and $Z^T = Z$.  Use this to deduce $Y^T = -Y$.  Since $\dU^1$ has no lowest weight elements of weight~$0$, $Y$ must be in $\dU^3$.  Hence $Q^{e_2} e_{-1}$ is in $\dU^3$ by~(iv).  To prove $(Q - \frac{1}{3}) e_{-1} \in \dU^3$, apply $\sym_3$ to the lowest weight element $e_{-1} e_0^2 - e_{-1}^2 e_1$ of $\S^3(\da)$.  This also implies~(iii).

Computation gives~(iv), but let us give a conceptual proof.  By Proposition~\ref{U3a}, $Y$ is the $0$-lowest weight element of an $\da$-copy of $\t\F_0$.  Therefore $\ad(e_1) Y$ is the $1$-lowest weight element of said $\t\F_0$.  By Proposition~\ref{U_3a}, $(\dU_3)_1^{e_{-1}}$ is $1$-dimensional, so it is spanned by $Q^{e_2} e_{-1}$.  This proves the result up to a scalar.  To compute the scalar, keep track only of terms involving $e_3$.
\end{proof}

\begin{rem}
$\t K_0$ is $\da$-generated by any $\ad(e_{-1})$-preimage of $Y$.  An explicit basis of the entire $0$-generalized eigenspace of $\dU_3$ may be found in Lemma~2.6 of \cite{CM07}.
\end{rem}

\begin{cor} \label{U3 LWVs}
The following elements of $\dU_3$ form a basis of the sum of its lowest weight spaces of weight~$\le 1$:
\begin{equation*}
   e_{-1}^3,\quad e_{-1}^2,\quad e_{-1},\quad Q e_{-1},\quad
   1,\quad Q,\quad Y,\quad Q^{e_2} e_{-1}.
\end{equation*}
They act on the universal \tdm\ $\F_\Lambda$ by, respectively, the operators
\begin{equation*}
   \px^3,\quad \px^2,\quad \px,\quad q(\Lambda) \px,\quad
   1,\quad q(\Lambda),\quad y(\Lambda),\quad 0.
\end{equation*}
\end{cor}

We now proceed toward partial analogs of Proposition~\ref{H2ell} and Lemma~\ref{ZQe2}.  Consider the following subspaces of $\dU^3$:
\begin{equation*}
   L_1 := K_1 \oplus K_2 \oplus K_3 \oplus \cdots, \qquad 
   L_0 := \t K_0 + L_1, \qquad
   L_{-1} := K_{-1} \oplus L_1.
\end{equation*}

\begin{prop} \label{L}
\begin{enumerate}
\item[(i)]
$L_1 = \dU^3 \cap \la Q^{e_2} \ra$.
\smallbreak \item[(ii)]
$L_1$, $L_0$, and $L_{-1}$ are $\VR$-submodules of $\dU^3$.
\smallbreak \item[(iii)]
As $\VR$-modules, $L_0 / L_1 \cong \F_0$ and $L_{-1} / L_1 \cong \F_{-1}$.
\end{enumerate}
\end{prop}

\begin{proof}
For~(i), Lemma~\ref{F action} implies that $L_1$ is contained in $\Ann_{\VR}(\F_\Lambda)$, which is $\la Q^{e_2} \ra$ by Theorem~\ref{I2}.  The converse follows from the fact that the non-zero operators in Corollary~\ref{U3 LWVs} are linearly independent.

Consider~(ii).  By~(i), $L_1$ is $\VR$-invariant.  Since $L_0$ is $\da$-invariant, applying a PBW basis of $\dU$ with $\da$ on the right, as in the proof of Proposition~\ref{H2ell}, shows that the weights of $\ad(\dU) L_0$ are all non-negative.  Therefore it is simply $L_0$.  The same procedure shows that $\ad(\dU) L_{-1}$ and $L_{-1}$ have the same $0$-weightspace, which implies that they are equal.  For~(iii), recall that any $\VR$-module $\da$-isomorphic to $\F_\lambda$ is $\VR$-isomorphic to it.
\end{proof}

\begin{cor} \label{L action}
$L_1$, $L_0$, and $L_{-1}$ have the following images under $\pi_\Lambda$:
\begin{equation*}
   \pi_\Lambda(L_1) = 0, \qquad \pi_\Lambda(L_0) = y(\Lambda) \bC[x],
   \qquad \pi_\Lambda(L_{-1}) = \big( q(\Lambda) - {\ts\frac{1}{3}} \big) \pi_\Lambda(\VR).
\end{equation*}
\end{cor}

\begin{proof}
For~$L_0$, take repeated $\ad(e_{-1})$-preimages of the actions of $Q$ and $Y$.  The rest is immediate from the preceding results.
\end{proof}

\section{Annihilators over $\VR$ generated in degrees 2 and~3} \label{F_lambda}

Here we use Proposition~\ref{multimachine} to prove the third result of \cite{CM07}.  We expect that our approach will lead to descriptions of the annihilators of the indecomposable modules composed of two \tdm s described in \cite{FF80}.

\begin{thm}[\cite{CM07}] \label{Il}
For $\lambda \not= 0$ or~$1$, $\Ann_{\VR}(\F_\lambda) = \big\la Q - q(\lambda), Y - y(\lambda) \big\ra$, and it is generated in degress~$2$ and~$3$.
\end{thm}

\begin{proof}
Proposition~\ref{multimachine} does not apply to $\F_\lambda$ directly, but rather, to its augmentation $\bC \oplus \F_\lambda$, where $\bC$ is the trivial module: we will prove that
\begin{equation} \label{C+F_lambda}
   \Ann_{\VR}(\bC \oplus \F_\lambda) =
   \big\la Q^{e_2}, (Q - q(\lambda)) e_{-1}, Y - (\lambda - \oh) Q \big\ra.
\end{equation}
The theorem will then be a consequence of the following observations:
\begin{itemize}
\item
$\Ann_{\VR}(\bC \oplus \F_\lambda)$ is of codimension~$1$ in $\Ann_{\VR}(\F_\lambda)$;
\smallbreak \item
$Q - q(\lambda) \not\in \Ann_{\VR}(\bC \oplus \F_\lambda)$;
\smallbreak \item
$\big\la Q^{e_2}, (Q - q(\lambda)) e_{-1}, Y - (\lambda - \oh) Q \big\ra \subset \big\la Q - q(\lambda), Y - y(\lambda) \big\ra$.
\end{itemize}

The ingredients for Proposition~\ref{multimachine} applied to $\bC \oplus \F_\lambda$ are
\begin{align*}
   & r := 2,\quad d_1 := 2,\quad d_2 := 3,\quad J_2 := \db,\quad J_3 := \dc, \\[4pt]
   & \J := \dU_1(\VR) \dU(\dc) + (\VR) \db =
   \spanrm \big\{ e_n e_{-1}^k, e_n e_0, e_{-1}^k: n, k \in \bN \big\}, \\[4pt]
   & I := \Ann_{\VR}(\bC \oplus \F_\lambda) \cap \dU_3.
\end{align*}

We must check that conditions~(a)-(d) of Proposition~\ref{multimachine} hold.  Recall the augmentation ideal $\dU^+$, the annihilator of $\bC$.  Note that
\begin{align*}
   & \pi_\lambda(e_n) = x^{n+1} \px + \lambda (n + 1) x^n, \\[4pt]
   & \pi_\lambda(e_n e_0 - e_n - e_{n+1} e_{-1}) = q(\lambda) (n+1) x^n.
\end{align*}
Together with a symbol argument, this shows that $\pi_\lambda$ is injective on $\J \cap \dU^+$, which verifies~(a).  Conditions (b) and~(c) are clear, and here (e) is vacuous.

In order to check (d) we must describe $I$ more explicitly.  Define
\begin{align*}
   & \t K_0(\lambda) := (\t K_0 \oplus G_0) \cap \Ann_{\VR}(\bC \oplus \F_\lambda), \\[4pt]
   & K_{-1}(\lambda) := (K_{-1} \oplus \VR) \cap \Ann_{\VR}(\bC \oplus \F_\lambda).
\end{align*}
The reader may prove the following lemma using Lemma~\ref{U3stuff} and Corollary~\ref{U3 LWVs}:

\begin{lemma} \label{I3}
Assume $\lambda \not= 0$ or~$1$.
\begin{enumerate}
\item[(i)]
$\t K_0(\lambda) \acong \t\F_0$, with lowest weight elements $Y - (\lambda - \oh) Q$ and $Q^{e_2} e_{-1}$.
\smallbreak \item[(ii)]
$K_{-1}(\lambda) = (Q - q(\lambda)) \VR \acong \F_{-1}$.
\smallbreak \item[(iii)]
$I = K_{-1}(\lambda) \oplus (\t K_0(\lambda) + L_1) \oplus H_2$, and $\t K_0(\lambda) \cap L_1 = K_1$.
\end{enumerate}
\end{lemma}

Condition~(d) is to be be checked in two degrees, $2$ and~$3$.  In degree~$2$, we need
\begin{equation*}
   \o I_2 + \db \VR = \S^2(\VR).
\end{equation*}
By Lemma~\ref{I3}, $I_2$ is the space $H_2$ from Proposition~\ref{H2ell}, so this was already proven in the verification of~(d) for Theorem~\ref{I2}.

In degree~$3$, we must prove that
\begin{equation*}
   \o I_3 + \dc^2 \VR = \S^3(\VR).
\end{equation*}
By Lemma~\ref{I3}, $\sym_3$ projects both $I$ and $L_{-1} + L_0$ to the unique subspace of $\S^3(\VR)$ that is $\da$-isomorphic to $\F_{-1} \oplus \t\F_0 \oplus \F_2 \oplus \cdots$.  On the other hand, it projects $\dc^2 \VR$ to a $\db$-copy of $\F_{-3}$.  Therefore (d) follows from~\eqref{S3b}.

At this point we have verified that Proposition~\ref{multimachine} applies to prove
\begin{equation*}
   \Ann_{\VR}(\bC \oplus \F_\lambda) =
   \big\la K_{-1}(\lambda) + \t K_0(\lambda) + L_1 + H_2 \big\ra.
\end{equation*}
Consider~\eqref{C+F_lambda}.  Clearly its left side contains its right side.  Define
\begin{equation*}
   M_0(\lambda) := \t K_0(\lambda) + L_1 + H_2, \qquad
   M_{-1}(\lambda) := K_{-1}(\lambda) \oplus L_1.
\end{equation*}
Use Propositions~\ref{H2ell} and~\ref{L} and Corollary~\ref{U3 LWVs} to prove:

\begin{lemma}
\begin{enumerate}
\item[(i)]
$L_1 \oplus H_2 = \dU_3 \cap \la Q^{e_2} \ra$.
\smallbreak \item[(ii)]
$M_0(\lambda)$ and $M_{-1}(\lambda)$ are $\VR$-submodules of $\dU_3$.
\smallbreak \item[(iii)]
As $\VR$-modules, $M_0(\lambda) / (L_1 \oplus H_2) \cong \F_0$ and $M_{-1}(\lambda) / L_1 \cong \F_{-1}$.
\end{enumerate}
\end{lemma}

Write $\I$ for the right side of~\eqref{C+F_lambda}.  It contains $L_1 \oplus H_2$ because it contains $Q^{e_2}$.  It contains $M_{-1}(\lambda)$, because $(Q - q(\lambda)) e_{-1}$ projects to a lowest weight element of $M_{-1}(\lambda) / L_1$, which is irreducible under $\VR$.

It remains only to prove that $\I$ contains $M_0(\lambda)$.  By Lemma~\ref{I3}, $Y - (\lambda - \oh) Q$ projects to a lowest weight element of $M_0(\lambda) / (L_1 \oplus H_2)$, but this does not suffice, because $\F_0$ is not $\VR$-irreducible.  However, it will suffice to prove that $\I$ contains an $\ad(e_{-1})$-preimage of $Y - (\lambda - \oh) Q$, because $\F_0$ is generated by its $1$-weightspace.

In order to write such an element explicitly, recall that $Z$ is an $\ad(e_{-1})$-preimage of $Q$ in $G_0$, and let $Y_1$ be any $\ad(e_{-1})$-preimage of $Y$ in $\t K_0$ (there is a line of such elements, parallel to $\bC Q^{e_2} e_{-1}$).  Set
\begin{equation*}
   X := Z \big( Y - (\lambda - \oh) Q \big) - \big( Q - q(\lambda) \big)
   \big( Y_1 - (\lambda - \oh) Z \big).
\end{equation*}
Since $\I$ contains $K_{-1}(\lambda)$, it contains $(Q - q(\lambda)) \dU^+$, and hence $X$.  It is immediate that $\ad(e_{-1}) X = q(\lambda) (Y - (\lambda - \oh) Q)$.  Since $q(\lambda) \not= 0$, 
\eqref{C+F_lambda} is proven.

To verify that $\Ann_{\VR}(\F_\lambda)$ is not generated simply in degree~$2$, note that
\begin{equation*}
   \Ann_{\VR}(\F_\lambda) \cap \dU_2 = \bC \big(Q - q(\lambda) \big) \oplus H_2
   = \Ann_{\VR}(\F_{1-\lambda}) \cap \dU_2,
\end{equation*}   
but $\Ann_{\VR}(\F_\lambda) \not= \Ann_{\VR}(\F_{1-\lambda})$: the former contains $Y - y(\lambda)$, while the latter contains $Y + y(\lambda)$.
\end{proof}

\section{Remarks and questions} \label{Remarks}

Here we elaborate on some ideas mentioned in the introduction, make some conjectures, and pose some more tentative questions.

One might ask whether or not Proposition~\ref{machine} applies to any modules of finite dimensional Lie algebras.  In fact, it applies to the Verma modules of $\dsl_2$.  We may illustrate this by showing that it applies to the $\F_\lambda$ restricted to $\da$, as they are dual to the Verma modules.  The reader may check that it suffices to take $d := 2$, $I := \bC (Q - q(\lambda))$, $J := \spanrm \{ e_{-1}, e_1 \}$, and
\begin{equation*}
   \J := \spanrm \big\{ e_1^{j_1} e_{-1}^{j_{-1}},\, e_1^{j_1} e_0 e_{-1}^{j_{-1}}:
   j_1,\, j_{-1} \in \bN \big\}.
\end{equation*}
Note that here $J$ is not a subalgebra of $\da$.  It would be amusing if Proposition~\ref{machine} applied to any modules of other finite dimensional simple Lie algebras.

Consider indecomposable $\db$-split modules of $\VR$ obtained by extending $\F_\lambda$ by $\F_\mu$.  These were classified in \cite{FF80} and have been studied by many authors.  They do not exist unless $(\lambda, \mu)$ is on the following list, in which case there is exactly one:
\begin{itemize}
\item
$(0, 1)$;
\smallbreak \item
$(\lambda, \lambda + \delta)$ for $\lambda \in \bC$ and $\delta = 2$, $3$, or $4$;
\smallbreak \item
$(-4, 1)$ and $(0, 5)$;
\smallbreak \item
$(-\frac{5}{2} + \rho, \frac{7}{2} + \rho)$ for $\rho = \pm \oh \sqrt{19}$.
\end{itemize}

These extensions are $\da$-split \iff\ $\F_\lambda$ and $\F_\mu$ have distinct $Q$-eigenvalues, i.e., $\lambda + \mu \not= 1$.  In particular, the extension of $\F_0$ by $\F_1$ is not $\da$-split.  Therefore, under $\da$ it is the module $\t\F_0$ occurring in Section~\ref{U3}.  It contains a $\VR$-submodule which is an extension of $\bC$ by $\F_1$.  This submodule is annihilated by $Q$ and $Y$ but not by $Z$, and we conjecture that in fact $Q$ and $Y$ generate its annihilator.

Another ``small'' extension is that of $\bC$ by $\F_2$, which occurs inside the extension of $\F_0$ by $\F_2$.  It is $\da$-split and is given by the Gel'fand-Fuchs cocycle: it is part of the coadjoint representation of the Virasoro Lie algebra.  The lowest weight elements which annihilate it were given in \cite{Ke19}: they include, in order of increasing degree,
\begin{equation*}
   \ad(S) Q^{e_2}, \qquad (Q - 2) e_{-1}, \qquad 2Y - 3Q, \qquad
   Q^{e_2} e_{-1}, \qquad Q (Q - 2),
\end{equation*}
but not $Q^{e_2}$.  We expect that Proposition~\ref{multimachine}, perhaps modified, can be used to deduce a minimal ``good'' generating set for the annihilator.  Clearly any such set must contain the degree~$2$ weight~$4$ element $\ad(S) Q^{e_2}$.  It would seem reasonable to guess the set of the first three lowest weight elements above as a candidate.

Coming from the direction of large modules rather than small ones, consider the ``universal difference~$\delta$ extension'', the direct sum of the extensions of $\F_\lambda$ by $\F_{\lambda + \delta}$ over all~$\lambda$, where $\delta$ is fixed at $2$, $3$, or~$4$.  Call its annihilator $\A_\delta$.  If one could describe $\A_\delta$, one might be able to get at the annihilators of the individual difference~$\delta$ modules by analyzing $\dU(\VR) / \A_\delta$.  This was the strategy used to describe $\Ann_{\VR}(\F_\lambda)$ in \cite{CM07}.

It is easy to see that $\A_\delta$ contains all lowest weight elements of weight~$> \delta$.  In particular, $\A_2$ and $\A_3$ both contain $\ad(S) Q^{e_2}$.  Recall from Proposition~\ref{U_3a} that, speaking imprecisely, there are two degree~$3$ weight~$3$ lowest weight elements.  Both are necessarily in $\A_2$.  One of them is $e_{-1} \ad(S) Q^{e_2}$; call the other one $X$.  It can be shown that $X$ is not in $\A_3$, implying that $\ad(S) Q^{e_2}$ does not generate $\A_2$.  We conjecture that $\ad(S) Q^{e_2}$ and $X$ generate $\A_3$, and that this can be proven using Proposition~\ref{multimachine} in degrees $2$ and~$3$, with $J_2 = \da$ and $J_3 = \db$.

Regarding the difference~$5$ and difference~$6$ extensions, both are opposite-dual pairs, so their annihilators are transpose pairs.  The annihilators of the difference~$5$ extensions are the same as those of the difference~$4$ extensions of $\F_1$ by $\F_5$ and of $\F_{-4}$ by $\F_0$.  A description of the difference~$6$ annihilator pair would be impressive.

Coming from the direction of generators rather than modules, one could study for example the ideals $\la \ad(S)^\ell Q^{e_2} \ra$ for $\ell \ge 1$.  Are they all distinct?  Are they the annihilators of any natural modules?

And one more question, related to the conjectures made at the end of \cite{SP23}: are the annihilators of $\bC$ and the $\F_\lambda$ the only primitive ideals of $\dU(\VR)$?

\section{Related work} \label{Related}

We conclude by discussing some related articles.  Note the following formulas:
\begin{equation} \label{sym formulas} \begin{array}{l}
   e_a e_b = \sym \big( e_a e_b + \oh (b - a) e_{a+b} \big), \\[4pt]
   e_a e_b^2 = \sym \big( e_a e_b^2 + (b - a) e_b e_{a+b}
   - {\ts\frac{1}{6}} a (b - a) e_{a+2b} \big).
\end{array} \end{equation}

\subsubsection*{Work of Sierra and Walton}
In \cite{SW14}, these authors proved that $\dU(\VR)$ is not Noetherian, resolving a conjecture of Dean and Small.  Their proof was non-constructive; in \cite{SW16} they gave a constructive proof, based on the annihilator of the universal Verma module $\F_\Lambda$ under the positive weight subalgebra $x^2 \VR$ of $\VR$.  They prove in Theorem~5.1 that $\Ann_{x^2 \VR} (\F_\Lambda)$ is generated by
\begin{equation*}
   g := e_1 e_5 - 4 e_2 e_4 + 3 e_3^2 + 2 e_6.
\end{equation*}

Let us relate $g$ to our description of $\Ann_{\VR}(\F_\Lambda)$.  By~\eqref{sym formulas}, $g$ is symmetric, as indeed it must be: otherwise, repeated applications of $\ad(e_{-1})$ would give an element of $\VR$.  This means we may work in $\S^2(\VR)$, where a short computation gives $\ad(e_{-1})^4 g = 24 Q^{e_2}$, as noted in Remark~5.15 of \cite{SW16}.  In fact, $g$ is determined up to a scalar: it must be in the $1$-dimensional space $(H_2)_6 \cap \dU(x^2 \VR)$, the linear combinations of the elements $\ad(e_1^{6-2\ell} S^{\ell+1}) e_{-1}^2$ with $\ell =1, 2, 3$ which have no $e_{-1} e_7$ or $e_0 e_6$ terms.

One also finds descriptions of $\Ann_{x^2 \VR} (\F_\lambda)$ in \cite{SW16}.  In Proposition~2.5 it is proven that for $\lambda = 0$ and~$1$ they are equal and are generated by
\begin{equation*}
   h_0 := e_1 e_3 - e_2^2 - e_4.
\end{equation*}
Again, this is symmetric and determined up to a scalar: $\ad(e_{-1})^3 h_0 = -24 Z$ is easily verified by working in $\S^2(\VR)$, and so $h_0$ must be an element of the $1$-dimensional space $(H_0)_4 \cap \dU(x^2 \VR)$, the linear combinations of the elements $\ad(e_1)^3 Z$ and  $\ad(e_1^{4-2\ell} S^{\ell+1}) e_{-1}^2$ with $\ell = 1, 2$ which have no $e_{-1} e_5$ or $e_0 e_4$ terms.

Regarding the case $\lambda \not = 0, 1$, note Remark~3.14 of \cite{SW16}: their parameter $a$ is our $1 - \lambda$.  They prove in Proposition~2.8 that $\Ann_{x^2 \VR} (\F_\lambda)$ is generated by three elements: $g$, which corresponds to $Q^{e_2}$;  a degree~$3$ weight~$5$ element $h_1$, and a degree~$4$ weight~$6$ element $h_3$.  Let us relate $h_1$ to our work.  It is
\begin{equation*}
   h_1 := e_1 e_2^2 - e_1^2 e_3
   + 2 (\lambda - 1) e_2 e_3 - (2\lambda - 3) e_1 e_4
   - (\lambda - 1) (\lambda - 2) e_5.
\end{equation*}
Use~\eqref{sym formulas} coupled with $e_b^2 e_a = -(e_a e_b^2)^T$ and Lemma~\ref{transpose parity} to deduce
\begin{equation*}
   h_1 = \sym \big( e_1 e_2^2 - e_1^2 e_3 + 2 (\lambda - \oh) (e_2 e_3 - e_1 e_4)
   - (q(\lambda) - {\ts\frac{1}{3}}) e_5 \big).
\end{equation*}
A not-too-long symbol computation in $\S^3 (\VR)$ gives
\begin{equation*}
   \ad(e_{-1})^5 (e_1 e_2^2 - e_1^2 e_3) = 240 (e_0^3 - e_{-1}^2 e_2);
\end{equation*}
it helps to observe that the action of $\ad(e_{-1})$ on $\S^k(\VR)$ preserves the space of linear combinations of monomials in the $e_i$ whose coefficients sum to zero.  The right side has the same symbol as $2 \cdot 5! (3Q e_0 - 2Y)$, and so Lemma~\ref{I3} implies
\begin{equation*}
   \ad(e_{-1})^5 h_1 = 6! \big( Q - q(\lambda) \big) e_0 - 4 \cdot 5! \big( Y - (\lambda - \oh) Q \big).
\end{equation*}
It can be shown that up to a scalar, $h_1$ is the only degree~$3$ element of $\Ann_{x^2 \VR} (\F_\lambda)$ of weight~$5$, and there are no such elements of lower weight.  We did not relate $h_3$ to our generators, but perhaps it satisfies similar uniqueness conditions.

\subsubsection*{Work of Billig and Futorny}
In \cite{BF16}, these authors make use of certain operators
\begin{equation*}
   \Omega^{(m)}_{k,s} := \sum_{a=0}^m (-1)^a \binom{m}{a} e_{k-a} e_{s+a}.
\end{equation*}
Let us consider them in the context of our Section~\ref{U2}.  They satisfy the binomial coefficient recursion $\Omega^{(m)}_{k,s} = \Omega^{(m-1)}_{k,s} - \Omega^{(m-1)}_{k-1,s+1}$, whence induction gives
\begin{equation*}
   \ad(e_{-1}) \Omega^{(m)}_{k,s} =
   (k+1-m) \Omega^{(m)}_{k-1 ,s} + (s+1) \Omega^{(m)}_{k, s-1}.
\end{equation*}

Note that $\Omega^{(m)}_{k,s}$ is in $\dU(\VR)$ if $k+1-m$ and $s+1$ are non-negative.  In this case, repeated applications of $\ad(e_{-1})$ map it eventually to $\Omega^{(m)}_{m-1, -1}$, and then to zero.  Consider $\Omega^{(m)}_{m-1, -1}$.  It follows from Lemma~\ref{S formula} that for $m \ge 3$ and odd, it is $0$, while for $m$ even, it is a non-zero multiple of the lowest weight element $\ad(S)^{m/2} e_{-1}^2$ of weight $m-2$.  Working from Proposition~\ref{H2ell}, one finds that
\begin{equation} \label{H Omega}
   H_{2\ell} = \spanrm \big\{ \Omega^{(m)}_{k,s}:\,
   k+1-m,\, s+1 \in \bN,\ m \ge 2\ell+1 \big\}.
\end{equation}

Suppose now that $V$ is a $\db$-trivial extension of $L$ tensor density modules, say $\F_{\lambda_1}, \ldots, \F_{\lambda_L}$, where $\lambda_1 \le \cdots \le \lambda_L$.  (Note that the Jordan-H\"older length of $V$ exceeds $L$ by the number of occurrences of $\F_0$.)  Let us refer to $\lambda_L - \lambda_1$ as $V$'s ``total jump''.  As intimated in Section~\ref{Remarks}, it is easy to prove that $H_{2\ell}$ annihilates $V$ whenever $2\ell$ exceeds the total jump.  The results of \cite{O'D18} lead us to conjecture that if $V$ is indecomposable, then it has total jump at most $2L + 2$, and so is annihilated by $H_{2L+4}$.

Corollary~3.4 of \cite{BF16} states that for every fixed $L$, there is an $m_L$ such that for all $(k, s)$, $\Omega^{(m_L)}_{k, s}$ annihilates all extensions with the same form as $V$.  In light of~\eqref{H Omega}, the preceding conjecture becomes $m_L = 2L+5$.  Example~3.5 of \cite{BF16} explains that for $L = 2$, this follows from \cite{FF80}.

\subsubsection*{Work of Petukhov and Sierra}
Finally, let us mention one result from \cite{PS20}, Theorem~1.7: any ideal in $\dU(x^2 \VR)$ containing a quadratic element is of finite Gel'fand-Kirillov codimension.  The $\VR$-analog would be that any ideal containing some $H_{2\ell}$ is of finite Gel'fand-Kirillov codimension.  Are there non-zero ideals containing no quadratic elements?  Consider Question~7.2 of \cite{CM07}: are the annihilators of the differential operator modules non-zero?  If so, they would be examples.

We also mention that \cite{PS20}, like \cite{SP23}, contains many interesting conjectures.

\def\eightit{\it} 
\def\bib{\bf}
\bibliographystyle{amsalpha}

\begin{thebibliography}{O'D18}

\bibitem[BF16]{BF16}
{\sc Y.~Billig, V.~Futorny}, Classification of simple $W_n$-modules with finite-dimensional weight spaces, {\it J.\ Reine. Angew.\ Math.\/} {\bf 720} (2016), 199--216.

\bibitem[CM07]{CM07}
{\sc C.~H.~Conley, C.~Martin}, Annihilators of tensor density modules, {\it J.\ Alg.\/} {\bf 312} (2007), no.~1, 495--526.

\bibitem[Du71]{Du71}
{\sc M.~Duflo}, Construction of primitive ideals in an enveloping algebra, in: I.~M.~Gel'fand, ed., {\it Publ. of 1971 Summer School in Math.,\/} Janos Bolyai Math.\ Soc., Budapest, 77--93.

\bibitem[FF80]{FF80}
{\sc B.~L.~Feigin, D.~B.~Fuchs}, Homology of the Lie algebra of vector fields on the line, {\it Funct.\ Anal.\ Appl.\/} {\bib 14} (1980), no.~3, 201--212.

\bibitem[Go23]{Go23}
{\sc W.~Goode}, {\it Annihilators of irreducible representations of the Lie superalgebra of contact vector fields on $\Roo$,} Ph.D.~Thesis, University of North Texas, 2023.

\bibitem[Ke19]{Ke19}
{\sc T.~Kenefake}, {\it Annihilators of bounded indecomposable modules of $\VR$,} Ph.D.~Thesis, University of North Texas, 2019.

\bibitem[Ma92]{Ma92}
{\sc O.~Mathieu}, Classification of Harish-Chandra modules over the Virasoro Lie algebra, {\it Invent.\ Math.\/} {\bib 107} (1992), no.~2, 225--234.

\bibitem[O'D18]{O'D18}
{\sc C.~O'Dell}, {\it Non-resonant uniserial representations of $\VR$,} Ph.D.~Thesis, University of North Texas, 2018.

\bibitem[PS20]{PS20}
{\sc A.~V.~Petukhov, S.~J.~Sierra}, Ideals in the enveloping algebra of the positive Witt algebra, {\it Algebr.\ Represent.\ Theory\/} {\bib 23} (2020), no.~4, 1569--1599.

\bibitem[SP23]{SP23}
{\sc S.~J.~Sierra, A.~V.~Petukhov}, The Poisson spectrum of the symmetric algebra of the Virasoro algebra, {\it Compos.\ Math.\/} {\bib 159} (2023), no.~5, 933--984.

\bibitem[Pi90]{Pi90} {\sc G.~Pinczon}, The enveloping algebra of the Lie superalgebra osp(1,2), {\it J.~Alg.\/} {\bib 132} (1990), 219--242.

\bibitem[SW14]{SW14}
{\sc S.~J.~Sierra, C.~Walton}, The universal enveloping algebra of the Witt algebra is not Noetherian, {\it Adv.\ Math.\/} {\bib 262} (2014), 239--260.

\bibitem[SW16]{SW16}
{\sc S.~J.~Sierra, C.~Walton}, Maps from the enveloping algebra of the positive Witt algebra to regular algebras, {\it Pacific J.\ Math.\/} {\bib 284} (2016), no.~2, 475--509.

\end{thebibliography}

\end{document}